\documentclass[11pt, bibother, otherbib]{asl-mod}

\usepackage{amssymb}
\usepackage{url, enumerate}

\usepackage{xpatch}
\usepackage{hyperref} 
 \usepackage{tikz}

\usepackage[T1]{fontenc} 
 
\fontencoding{T1}\selectfont
 
\usepackage{footnotebackref}

\makeatletter
\xpretocmd{\@maketitle}{\let\@makefntext\BHFN@OldMakefntext}{}{}
\renewcommand\@makefntext[1]{%
  \@ifundefined{@makefnmark}
    {}
    {%
     \renewcommand\@makefnmark{%
       \mbox{%
         \textsuperscript{%
           \normalfont
           \hyperref[\BackrefFootnoteTag]{\@thefnmark}%
         }%
       }\,%
     }%
     \BHFN@OldMakefntext{#1}%
  }%
}
\makeatother

\newenvironment{proof-}{\noindent{\bf Proof}}{\par\bigskip}

\newtheorem{Theorem}{Theorem}[section]
\newtheorem{theorem}{Theorem}[section]

\newtheorem{lemma}[Theorem]{Lemma}

\newtheorem{claim}[Theorem]{Claim}

\newtheorem{corollary}[Theorem]{Corollary}
\newtheorem{observation}[Theorem]{Observation}

\theoremstyle{definition}

\newtheorem{definition}[Theorem]{definition}
\newtheorem{notation}[Theorem]{Notation}
\newtheorem{remark}[Theorem]{Remark}

\newcommand{\hgt}{{\rm ht}}
\newcommand{\lev}{{\rm lev}}

\newcommand{\into}{\rightarrow}

\newcommand{\rest}{\upharpoonright}  

\newcommand{\DD}{{\mathcal D}}
\newcommand{\EE}{{\mathcal E}}
\newcommand{\FF}{{\mathcal F}}

\newcommand{\HH}{{\mathcal H}}

\newcommand{\LL}{{\mathcal L}}

\newcommand{\TT}{{\mathcal T}}

\def\empty{\emptyset}

\newcount\skewfactor
\def\mathunderaccent#1#2 {\let\theaccent#1\skewfactor#2
\mathpalette\putaccentunder}
\def\putaccentunder#1#2{\oalign{$#1#2$\crcr\hidewidth
\vbox to.2ex{\hbox{$#1\skew\skewfactor\theaccent{}$}\vss}\hidewidth}}

\newcommand{\dom}{\mbox{\rm dom}}

\newcommand{\ran}{\mbox{\rm ran}}

\author{Mirna D\v zamonja}
\revauthor{Mirna D\v zamonja}
\address{Institut d'Histoire et de Philosophie des Sciences et des Techniques,
CNRS-Universit\' e Paris 1 Panth{\'e}on-Sorbonne,
13 Rue de Four, 75006 Paris, France
and
Institut of Mathematics,
Czech Academy of Sciences,
{\v Z}itn{\'a} 25, 115 67 Prague, Czech Republic}\email{mirna.dzamonja@univ-paris1.fr, https://www.logiqueconsult.eu}
\author{Saharon Shelah}
\revauthor{Saharon Shelah}
\address{Department of Mathematics, Hebrew University of Jerusalem, 91904 Givat Ram, Israel}
\email{shelah@math.huji.ac.il, http://shelah.logic.at}



\title{On wide Aronszajn trees in the presence of MA}

\begin{document}
\maketitle

\begin{abstract} A wide Aronszajn tree is a tree of size and height $\omega_1$ with no uncountable branches. We prove that under $MA(\omega_1)$ there is no wide Aronszajn tree which is universal under weak embeddings. This solves an open question of Mekler and
V\"a\"an\"anen from 1994. 

We also prove that under $MA(\omega_1)$, every wide Aronszajn tree weakly embeds in an Aronszajn tree, which combined with a result of Todor{\v c}evi{\'c} from 2007, gives that under  $MA(\omega_1)$ every wide Aronszajn tree embeds into a Lipschitz tree or a coherent tree. We also prove that under $MA(\omega_1)$  there is no wide Aronszajn tree which weakly embeds all Aronszajn trees, improving the result in the first paragraph as well as a result of Todor{\v c}evi{\'c}  from 2007 who proved that under $MA(\omega_1)$ there are no universal Aronszajn trees.  
\end{abstract}

\footnote{Mirna D\v zamonja's research was supported by the GA{\v C}R project EXPRO 20-31529X and RVO: 67985840. She thanks the University of East of Anglia where she was Professor when this research was done and Hebrew University of Jerusalem for their hospitality in April 2019. Saharon Shelah thanks the Israel Science Foundation for their grant 1838/19 and the European Research Council for their grant 338821. The authors thank Tanmay Inamdar for many very constructive comments and the anonymous referee for a prompt and careful reading and, moreover, providing them with the enclosed diagrams in Latex. They would also like to acknowledge that after reading this preprint, Stevo Todor{\v c}evi{\'c} informed them that upon closer inspection one can see that  his arguments regarding Lipschitz trees in \cite{To} reveal that the fact that the trees have countable levels is not used, and therefore one could use 
 \cite{To} combined with \cite{stevo2007walks} to solve the Mekler-V\"a\"an\"anen problem.
 
This is publication number 1186 in Shelah's publication list.
}

\section{Introduction}
We study the class $\TT$ of trees of height and size $\aleph_1$, but with no uncountable branch. We call such trees {\em wide Aronszajn trees}.
A particular instance of such a tree is a classical Aronszajn tree, so the class $\mathcal A$ of Aronszajn trees satisfies 
$\mathcal A\subseteq \TT$. Apart from their intrinsic interest in combinatorial set theory, these classes are also interesting from the topological point of view, since they give rise to a natural generalisations of metric spaces, 
$\omega_1$-metric spaces introduced by Sikorski in \cite{SikorskiDelta} and further studied in \cite{MekVa}, \cite{To} or \cite{firstwithJouko}, for example. The $\omega_1$-distance function in trees is given by the $\Delta$-function,
which is defined by $\Delta(x,y)=\hgt(x\cap_T y)$ for $x\neq y$ and $\Delta(x,x)=0$. Here $x\cap_T y$ represents the meet in the tree (as our trees will be trees of sequences of ordinals, this notation is more natural than $\wedge_T$).
Classes $\TT$ and $\mathcal A$ can be quasi-ordered using the notion of {\em weak embedding}, which is defined as follows:

\begin{definition}\label{} For two trees $T_1, T_2$, we say that $T_1$ is {\em weakly embeddable} in $T_2$
and we write $T_1\le T_2$,
if there is $f:\,T_1\to T_2$ such for all $x, y \in T_1$
\[x <_{T_1} y\implies f(x) <_{T_2} f(y).
\]
\end{definition}

We are interested in the structure of $(\TT, \le)$ and  $(\mathcal A, \le)$. In particular, we address the question of the
existence of a universal element in these classes. This is of special interest since among the many interesting and correct results of the paper \cite{MekVa}
from 1993 there is also a claim that $MA(\omega_1)$ implies that there is a universal element in $(\TT, \le)$, the argument for which was soon after found to be faulty. Ever since, the status of the possible
existence of a universal element in  $(\TT, \le)$ under $MA(\omega_1)$ has remained an open question.

Our first result is Theorem \ref{nounivAron}, which proves that under $MA(\omega_1)$ there is no universal element 
in $(\mathcal A, \le)$. This gives an alternative proof to a result of Todor{\v c}evi{\'c} from \cite{stevo2007walks}, whose Theorem 4.3.34 proves the same using the class of coherent trees. For more on this see \S\ref{history}.

The second result is Theorem \ref{wideintonormal}, which shows that under
$MA(\omega_1)$ every wide Aronszajn tree weakly embeds into an Aronszajn tree. Putting the two results together, we obtain the main result of the paper,  Theorem \ref{main}, which shows that under $MA(\omega_1)$ the class
$(\TT, \le)$ has no universal element. This resolves the question raised by \cite{MekVa}.  

Combining our result with Lemma 4.3.32 from \cite{stevo2007walks}, we obtain that under $MA(\omega_1)$ every wide Aronszajn tree weakly embeds into a coherent tree, or equivalently under $MA(\omega_1)$, into a Lipschitz tree
(Corollary \ref{improve}(1)). We also obtain (Corollary \ref{improve}(2)) a strengthening of Todor{\v c}evi{\'c}'s result about the non-existence of universal Aronszajn trees under $MA(\omega_1)$, namely we prove that under $MA(\omega_1)$ not even the class of wide Aronszajn trees suffices to weakly embed all Aronszajn (or all coherent) trees.
 
\section{Some facts about $(\TT, \le)$ and  $(\mathcal A, \le)$}\label{history}
Note that if there is a weak embedding from a tree to another, then there is one which preserves levels (see Observation \ref{levelpreserving}), so we may restrict our attention to such embeddings.

An important idea of {\DJ}uro Kurepa in \cite{Kurepasigma} (see \cite{Kurepacompletesigma} for a complete edition) is that of
a functor now known as $\sigma$-functor. This functor associates to a tree $T$ the tree $\sigma T$ of the increasing sequences of $T$, ordered by inclusion. The basic fact is that there cannot be a weak embedding from $\sigma T$ to $T$.
If $T$ has no uncountable branch, neither does $\sigma T$, but even if the cardinality of $T$ is $\aleph_1$, the
cardinality of $\sigma T$ is $2^{\aleph_0}$. However, when $CH$ holds, for any $T\in \TT$ we have 
$\sigma T \in \TT$ and similarly for $\mathcal A$. Therefore, under $CH$ neither class $(\TT, \le)$ nor $(\mathcal A, \le)$ have a universal element.

Todor{\v c}evi{\'c} studied level-preserving functions $f$ between trees which satisfy the Lipschitz condition
\begin{equation}\label{Lipschitzcond}
\Delta_{T_1}(x,y)\le \Delta_{T_2}(f(x), f(y)).
\end{equation}
We may think of Lipschitz embeddings as contractions. This notion led Todor{\v c}evi{\'c} to introduce a subclass 
of $\mathcal A$ which consists of those Aronszajn trees on which every level-preserving map from an uncountable subset of $T$ into $T$, has an uncountable Lipschitz restriction. These are called Lipschitz trees. After an initial 1996 preprint with many properties of Lipschitz trees, including the shift operation $T^{(1)}$, the full paper by Todor{\v c}evi{\'c} on this topic appeared as \cite{To}. 
In particular, by considering embeddings between Aronszajn trees into Lipschitz ones, the paper proves that assuming $BPFA^{\aleph_1}$, there is no universal element in $(\mathcal A, \le)$. Finally, in his book \cite{stevo2007walks} Todor{\v c}evi{\'c} studies the class of coherent trees, which are Aronszajn trees obtained from ordinal walks, and he proves that under $MA(\omega_1)$ all coherent trees are Lipschitz and
that such a tree $T$ embeds into $T^{(1)}$ but not the other way around. Moreover, still under $MA(\omega_1)$ every Aronszajn tree embeds into a coherent tree. This leads to the conclusion, Theorem 4.3.44 in \cite{stevo2007walks}:

\begin{theorem}[Todor{\v c}evi{\'c}]\label{Toto}(\cite{To}) Assuming $MA(\omega_1)$, there is no universal element in $(\mathcal A, \le)$.
\end{theorem}

Many more results are known about $(\mathcal A, \le)$, one can consult surveys \cite{Stevoontrees} for earlier and \cite{JustinstructureA} for more recent results. 

Not that much is known about the full class $(\mathcal T, \le)$. We cite the two results that we are aware of. The first one is a consistency result
obtained by Mekler and V\"a\"an\"anen.

\begin{Theorem}\label{con} (\cite{MekVa}) Assume $CH$ holds and $\kappa$
is a regular cardinal satisfying $\aleph_2\le\kappa$
and $\kappa\le 2^{\aleph_1}$. Then there is
a forcing notion that preserves cofinalities (hence cardinalities) and the value of
$2^\lambda$ for all $\lambda$, and which forces the universality number of $({\TT},\le)$ 
and the universality number of $(\mathcal A, \le)$ both to be $\kappa$.
\end{Theorem}

The next result, obtained by 
D{\v z}amonja and V\"a\"an\"anen,
 is in the presence of club guessing at $\omega_1$ and the failure of $CH$. It concerns weak embeddings
called $\Delta$-preserving and defined by 
 \begin{equation}
\Delta_{T_1}(x,y)= \Delta_{T_2}(f(x), f(y)).
\end{equation}

\eject 

\begin{Theorem}[D{\v z}amonja and V\"a\"an\"anen](\cite{firstwithJouko})\label{result} Suppose that
\begin{description}
\item{(a)} there is a ladder system $\bar{C}=\langle c_\delta:\,\delta<\omega_1\rangle$
which guesses clubs, i.e. satisfies that for any club $E\subseteq \omega_1$ there 
are stationarily many
$\delta$ such that $c_\delta\subseteq E$,
\item{(b)} $\aleph_1 < 2^{\aleph_0}$.
\end{description}
Then no family of
size $<2^{\aleph_0}$ of trees of size
$\aleph_1$, even if we allow uncountable branches, can $\le$-embed all members of
${\TT}$ in a way that preserves $\Delta$. 
\end{Theorem}

Before this paper it was not known if $(\mathcal T, \le)$ had a universal element under $MA(\omega_1)$. Our result
\ref{main} proves that it does not. It is not known if there is a model of set theory in which $(\mathcal T, \le)$ does have a 
universal element. Moreover, our results (see Corollary \ref{improve}(2)) strengthen both this conclusion and Theorem
\ref{Toto} in that they imply that under $MA(\omega_1)$ there is no $T$ in the larger class $(\mathcal T, \le)$
which weakly embeds all elements of $(\mathcal A, \le)$. It is not known if there is a model of set theory in which
$(\mathcal A, \le)$ or $(\mathcal T, \le)$ have a universal element.

\section{Specialising triples and their basic properties} 

\begin{notation}\label{} (1) For an ordinal $\gamma <\omega_1$ we denote by $\hgt(\gamma)$ the unique $\alpha$ such that 
$\gamma\in [\omega\alpha, \omega\alpha +\omega)$.

\smallskip 
{\noindent (2)} 
We can without loss of generality represent $\mathcal A$ as the set of all normal rooted $\omega_1$-trees $T$ with no uncountable branches whose
$\alpha$-th level is indexed by a subset of the ordinals in 
$[\omega\alpha, \omega\alpha +\omega)$, for $\alpha<\omega_1$.
The root $\langle \rangle$ is considered of level $-1$. 

(Recall that the requirement of being {\em normal} for a rooted tree means that
if $\gamma_0\neq \gamma_1$ are of the same limit level, then there exists $\beta$ with $\beta <_T\gamma_l$ 
for exactly one $l<2$).

\smallskip 
{\noindent (3)} If $T\in \mathcal A$ and $s,t \in T$,
we denote by $s\cap_T t$ the maximal ordinal $\gamma$ such that $\gamma <_T s, t$. (Such an ordinal exists by the assumption in (2)).

If $\hgt(x)=\alpha >\beta$, then by $x\rest\beta$ we denote the unique ordinal $y$ with 
$\hgt(y)=\beta$ and $y<_T x$.

\smallskip 
{\noindent (4)} For $T_1, T_2 \in \mathcal A$ and $(x,y)\in 
\bigcup_{\alpha <\omega_1} \lev_\alpha(T_1)\times \lev_\alpha (T_2)$, we let $\alpha(x,y)$ denote the $\alpha$ such that
$x\in  \lev_\alpha (T_1)$ (and so $y\in  \lev_\alpha(T_2)$).
\smallskip 
\end{notation}

\begin{definition}\label{the_triples} 
Let $\mathcal A_2^{\rm sp}$ be the set of all triples $(T_1, T_2, c)$ where $T_1, T_2 \in \mathcal A$ and  $c$ is a function from
$\bigcup_{\delta \mbox{  limit }<\omega_1} \lev_\delta (T_1)\times \lev_\delta (T_2) $ to $\omega$ such that
\begin{itemize}
\item\label{switchinglevels} if $c(x_1, y_1)= c(x_2, y_2)$ and $(x_1, y_1)\neq (x_2, y_2)$, then
 $\alpha(x_1, y_1)\neq \alpha (x_2, y_2)$, $x_1\bot _{T_1}x_2$, 
 $y_1\bot_{T_2} y_2$ and 
 \[
 \Delta_{T_1}(x_1, x_2)> \Delta_{T_2}(y_1, y_2).  
 \]
\end{itemize}
\end{definition}

\begin{remark}\label{} By the definition of $\mathcal A$, we have that for any $T\in \mathcal A$ and any $\gamma\in T$,
$\hgt(\gamma)$ is the same as $\hgt_T(\gamma)$. The defining condition of specialising triples could have therefore been
written in termes of heights, $\hgt(x_1\cap x_2)> \hgt(y_1\cap y_2)$.

Also note that a weak embedding is not required to be injective, but is injective on any branch of its domain. Finally,
observe that every rooted Aronszajn tree is weakly bi-embeddable with a rooted normal one and hence that concentrating on rooted normal trees does not change anything from the point of view of universality results.
\end{remark}

The following is well known, see for example Claim 6.1 of \cite{MirnaJoukochains}.

\begin{observation}\label{levelpreserving} If there exists a weak embedding from a tree $T_1$ to a tree 
$T_2$, then there exists one which preserves levels, namely satisfying $\hgt_{T_1}(x)=\hgt_{T_2}(f(x))$
for all $x\in T_1$.
\end{observation}

\begin{proof} Let $f:\, T_1\to T_2$ be a weak embedding. For $t\in T_1$, we can define $g(t)=f(t)\rest \hgt(t)$,
since $f$ being a weak embedding implies that for every such $t$ we have $\hgt_{T_1}(t)\le \hgt_{T_2}(f(t))$.
Now note that if
$s<_{T_1} t$, then $\hgt(s)<_{T_1} \hgt(t)$ and so $g(s)<_{T_2}g(t)$.
\end{proof}

\begin{claim}\label{specialisation} (1) If $(T_1, T_2, c) \in \mathcal A_2^{\rm sp}$ then both $T_1$ and $T_2$ are special Aronszajn trees.

\smallskip 

{\noindent (2)}  If $(T_1, T_2, c)\in  \mathcal A_2^{\rm sp}$ then $T_1$ is not weakly embeddable in $T_2$.

\smallskip 

{\noindent (3)} Every rooted normal Aronszajn tree is isomorphic to a tree in $\mathcal A$.
\end{claim}

\begin{proof} (1) Clearly, every tree in $\mathcal A$ is an $\omega_1$-tree, so $T_1$ and $T_2$ are $\omega_1$-trees. 
Let us first show that $T_1$ is special, so we shall define a function $d:\,T_1\to \omega$ which witnesses that. 

Notice that by the assumption that $T_2$ is of height $\omega_1$, we can choose $z_\delta$ of height $\delta\in T_2$,
for every limit $\delta$. Let $g:\,\omega\times\omega\times\omega\to\omega$
be a bijection. Every $x\in T_1$ is of the form $\omega\delta+\omega m+ n$ for some limit ordinal $\delta$ and natural numbers $m$ and $n$. For such $x$, define $d(x)= g(c(x\rest\delta, z_\delta), m,n)$.

Suppose that $x=\omega\delta+\omega m+ n$, $y=\omega\beta+\omega k+ l$ and that $d(x)=d(y)$, while $x\neq y$.
Therefore $g(c(x\rest\delta, z_\delta), m,n)=g(c(y\rest\beta, z_\beta), k,l)$ and we obtain $m=k$ and
$n=l$ while $c(x\rest\delta, z_\delta)=c(y\rest\beta, z_\beta)$. Since $x\neq y$ we must have $\beta\neq\delta$ and
therefore $x\rest\delta\neq y\rest \beta$.
By the properties of $c$ we obtain
$x\rest\delta\bot_{T_1} y\rest \beta$ and therefore $x\bot_{T_1} y$. In conclusion, $d^{-1}(\{a\})$ is an antichain, for any $a<\omega$, and therefore $d$ witnesses that $T_1$ is special. A similar proof shows that $T_2$ is special. As clearly every special $\omega_1$-tree is Aronszajn, the claim is proved.

\smallskip

{\noindent (2)} Suppose for a contradiction that $f$ is a weak embedding from $T_1$ to $T_2$. By Observation
\ref{levelpreserving}, we can assume that $f$ preserves levels. For each $\alpha$ limit $<\omega_1$ choose
$x_\alpha$ on the $\alpha$-th level of $T_1$. Note that by the level preservation of $f$, the value $c(x_\alpha, f(x_\alpha))$ is well-defined. Consider $\{ c(x_\alpha, f(x_\alpha)):\, \alpha \mbox{ limit }<\omega_1\}$, which is necessarily a countable set since the range of $c$ is $\omega$. Hence, there must be $\alpha<\beta$ such that $c(x_\alpha, f(x_\alpha))=c(x_\beta, f(x_\beta))$. By the defining property of $c$ we have that $x_\alpha\bot_{T_1} x_\beta$. 

Since $f$ is strict-order preserving we have that 
\[
f(x_\alpha \cap_{T_1} x_\beta)<_{T_2} f(x_\alpha), f(x_\beta)
\]
and therefore $f(x_\alpha \cap_{T_1} x_\beta)\le _{T_2} f(x_\alpha)  \cap_{T_2} f(x_\beta)$. However, 
\[
\hgt (f(x_\alpha \cap_{T_1} x_\beta)) =\hgt (x_\alpha \cap_{T_1} x_\beta) >\hgt ( f(x_\alpha)  \cap_{T_2} f(x_\beta)),
\]
a contradiction.

\smallskip

{\noindent (3)} Obvious.
\end{proof}

\section{Embeddings between Aronszajn trees and the non-existence of a universal element under $MA$}\label{sec:Aron} This section is devoted to the proof of the following theorem.

\begin{theorem}\label{nounivAron} For every tree $T\in\mathcal A$, there is a ccc forcing $\mathbb Q=\mathbb Q(T)$
and a family $\FF=\FF(T)$ of $\aleph_1$-many dense sets in $\mathbb Q$ such that every 
$\FF$-generic filter adds a
tree $T^\ast$ in $\mathcal A$ and a function $c$ such that $(T^\ast, T, c)$ form a specialising triple. In particular, $T^\ast$ is not weakly embeddable into $T$ and, hence, under the assumption of $MA(\omega_1)$ there is no Aronszajn tree universal under weak embeddings.
\end{theorem}

The latter is a result of Todor{\v c}evi{\'c}, see Theorem \ref{Toto}, to which our method gives an alternative proof.
We shall break the proof of Theorem \ref{nounivAron}  into the definition of the forcing and then several lemmas needed to make the desired conclusion.

\begin{definition}\label{defQ}
Suppose that $T \in \mathcal A$, we shall define a forcing notion $\mathbb Q=\mathbb Q(T)$ to consist of all $p=(u^p, v^p, <_p, c^p)$
such that:
\begin{enumerate}
\item $u^p\subseteq \omega_1\cup \{\langle\rangle\}$, $v^p\subseteq T$ are finite and $\langle\rangle\in v^p$,
\item if $\alpha\in v^p$ then there is $\beta\in u^p$ with $\hgt(\alpha)=\hgt(\beta)$,
\item $<_p$ is a partial order on $u^p$ such that $\alpha<_p\beta$ implies $\hgt(\alpha)<\hgt(\beta)$
and which fixes 
$\alpha\cap_{<_p} \beta\in u^p$ for every two different elements $\alpha, \beta$ of $u^p$ and fixes 
the root $\langle\rangle$ of $u^p$,
\item\label{mainpointQ} $c^p$ is a function from 
$\bigcup_{\delta \mbox{  limit }<\omega_1} \lev_\delta (u^p)\times \lev_\delta (v^p) $ to $\omega$ such that the analogue of
the requirement from Definition \ref{switchinglevels} holds, that is:

if $c(x_1, y_1)= c(x_2, y_2)$ and $(x_1, y_1)\neq (x_2, y_2)$, then
$\alpha(x_1, y_1)\neq \alpha (x_2, y_2)$, $x_1\bot _{T_1}x_2$, 
 $y_1\bot_{T_2} y_2$ and 
 \[
 \hgt (x_1\cap_{T_1} x_2) >  \hgt (y_1\cap_{T_2} y_2). 
 \]
\end{enumerate}

The order $p\le q$ on $\mathbb Q$ is given by inclusion $u^p\subseteq u^q, v^p\subseteq v^q,
<_p\subseteq <_q, c^p\subseteq c^q$ with the requirement that if $p\le q$, then the intersection and the root given by $<_p$ are preserved in $<_q$.
\end{definition}

\begin{lemma}\label{genericc} There is a family $\FF$ of $\aleph_1$-many dense subsets of $\mathbb Q$ such that for any $G$ which is $\FF$-generic, letting
\[
T^\ast=\bigcup\{<_p:\,p\in G\} \mbox{ and }  c=\bigcup\{c^p:\,p\in G\}
\]
gives $(T^\ast, T, c)\in \mathcal A_2^{\rm sp}$.
\end{lemma}

\begin{proof} Clearly, for any filter $G$ we have that $T^\ast$ is a partial order on $\omega_1$. 
For every $\alpha<\omega_1$ we have that $\lev_\alpha(T^\ast)\subseteq [\omega\alpha, 
\omega\alpha +\omega)$, since the same is true for every $<_p$ for $p\in G$. In particular, $T^\ast$ is a tree.
It is a rooted tree since every $u^p$ for $p\in G$ has the same root. Let us observe that $T^\ast$ is normal, using the
following claim.

\begin{claim}\label{normality} Suppose that $\beta_0,\beta_1\in [\omega\delta, \omega\delta+\omega)\cap T^\ast$,
where $\delta$ is a limit ordinal. Then there is $\alpha \in T^\ast$ such $\alpha<^\ast\beta_l$ for exactly one $l<2$.
\end{claim}

\begin{proof} We can find $p\in G$ such that $\beta_0, \beta_1\in u^p$. Therefore $<_p$ fixes $\beta=\beta_0\cap_{<_p}
\beta_1$ and by the definition of the order in $\mathbb Q$ we must have $\beta=\beta_0\cap_{<^\ast}
\beta_1$. Any $\alpha<^\ast\beta_1$ with $\hgt(\alpha)>\hgt(\beta)$ satisfies the requirement.
\end{proof}

We now show that with a judicious choice of $\FF$ we have that $T^\ast$ is of height $\omega_1$.

\begin{claim}\label{density1} For every $\alpha<\omega_1$, the set $\DD_\alpha$ of all $p$ such that $u^p$ has an element on level $\alpha$ is dense.
\end{claim}

\begin{proof} Given $\alpha<\omega_1$, if $u^p$ has no elements on level $\alpha$, we shall first 
choose a $\gamma\in [\omega\alpha, \omega\alpha+\omega)$ and extend the order $<_p$ to 
$u^p\cup\{\gamma\}$ by letting
$\gamma$ be above the root $\langle\rangle$ of $u^p$ but such that $\beta\cap_{<^p} \gamma=\langle\rangle$ for all $\beta\in u^p$.
Since $u^p$ did not have any elements on level $\alpha$, neither does $v^p$, so we do not have to worry about extending $c$ to include pairs whose first coordinate is $\gamma$.
\end{proof}

We can conclude that $T^\ast$ is a normal $\omega_1$-tree. The next density claim will show that 
$c$ is defined on all $\bigcup_{\delta \mbox{  limit }<\omega_1} \lev_\delta (T^\ast)\times \lev_\delta (T) $ to $\omega$ and will therefore by Claim \ref{specialisation} (1) imply that $T^\ast\in \mathcal A$.

\begin{claim}\label{density2} Suppose that $\delta \mbox{  is a limit ordinal }<\omega_1$ and that there is $x$ of height 
$\delta$ in $u^p$. If $y\in T$ is of height $\delta$, then $p$ has an extension $q$ such that $y\in v^q$, in other words, the set $\EE_y=\{q:\,y\in v^q\}$ is dense above $p$ .
\end{claim}

\begin{proof} It suffices to let $v^q=v^p\cup \{y\}$ and to extend $c^p$ in a one-to-one way so that for any 
$x\in u^p$ of height $\delta$, the value
of $c^q(x,y)$ is different from any values taken by $c^p$.
\end{proof}

Let $\FF$ consist of all sets $\DD_\alpha$ for $\alpha<\omega_1$ and all sets $\mathcal E_y$ defined in and Claim \ref{density2}.

To finish the proof of Lemma $\ref{genericc}$ we have that $c$ is as required, since every $p$ satisfies the requirement from \ref{defQ}(4).
\end{proof}

\begin{lemma}\label{ccc} The forcing $\mathbb Q(T)$ is ccc.
\end{lemma}

\begin{proof} Suppose that $\langle p_\zeta:\,\zeta <\omega_1\rangle$ is a given sequence of elements of
$\mathbb Q(T)$. By extending each $p_\zeta$ if necessary, we can assume that for each $\zeta$ there is
an element of $v^{p_\zeta}$ and hence of  $u^{p_\zeta}$ of height $\zeta$. Let 
$C=\{\zeta <\omega_1:\, \omega\zeta=\zeta\}$, so a club of $\omega_1$.

For $\zeta\in C$ let us define $q_\zeta=p_\zeta\rest \zeta$, by which we mean:
\begin{enumerate}
\item $u^{q_\zeta}=u^{p_\zeta}\cap (\zeta\cup\{\langle\rangle\}), v^{q_\zeta}=v^{p_\zeta}\cap (\zeta\cup\{\langle\rangle\})$, 
\item $<_{q_\zeta}=<_{p_\zeta}\rest u^{q_\zeta}$ and
\item $c^{q_\zeta}=c^{p_\zeta} \rest ( u^{q_\zeta} \times v^{q_\zeta})$.
\end{enumerate}
There is a stationary set $S\subseteq C$, a condition $q^\ast$ and integers $n^\ast, m^\ast <\omega$ such that 
for every $\zeta\in S$ we
have:
\begin{enumerate}
\item  $q_\zeta=q^\ast$,
\item the size of $u^{p_\zeta}\setminus u^{q^\ast}$ is $n^\ast$ and  the size of $v^{p_\zeta}\setminus v^{q^\ast}$ is $m^\ast$. We enumerate
them increasingly as  ordinals in the form $\langle x^\zeta_i:\, i<n^\ast\rangle$ and $\langle y^\zeta_j:\, j<m^\ast\rangle$,
\item the value of $c^{p_\zeta} (x^\zeta_i,y^\zeta_j)$ and the fact that it is defined or not depends only on $i$ and $j$ and not on $\zeta$, and
\item letting $\gamma^\ast=\max (u^{q^\ast}\cup v^{q^\ast})$, we have $\min(u^{p_\zeta}\setminus u^{q^\ast} )>
\gamma^\ast+\omega$ and similarly for $v^{p_\zeta}\setminus v^{q^\ast}$.
\end{enumerate}
By thinning further, we may assume that for every $\varepsilon <\zeta$ in $S$,
\begin{itemize}
\item $u^{p_\varepsilon}\cup v^{p_\varepsilon}
\subseteq \zeta$,
\item the unique ordinal-order-preserving functions $f_{\varepsilon, \zeta}$ from $u^{p_\varepsilon}$ 
to $u^{p_\zeta}$ and $g_{\varepsilon, \zeta}$ from $v^{p_\varepsilon}$ 
to $v^{p_\zeta}$ give rise to an isomorphism between $p_\varepsilon$ to $p_\zeta$ which fixes $q^\ast$. In particular, it maps
$<_{p_\varepsilon}$ to $<_{p_\zeta}$ fixing $u^{q^\ast}$ and similarly for $<_T\rest v^{p_\varepsilon}$
and $<_T\rest v^{p_\zeta}$.
\item for every $\alpha\in v^{p_\zeta}\setminus v^{p_\varepsilon}$ we have that $\alpha\rest_{T} (\gamma^\ast+\omega)= g_{\varepsilon, \zeta}^{-1} (\alpha)\rest_T  (\gamma^\ast+\omega)$.
\end{itemize}
Let us now consider what could render two conditions $p_\varepsilon$ and $p_\zeta$ for $\varepsilon$ and $\zeta$ in $S$, incompatible. The minimum requirement on a condition $q$ with $q\ge p_\varepsilon, p_\zeta$
would be that $u^q\supseteq u^{p_\varepsilon}\cup u^{p_\zeta}$ and $v^q\supseteq v^{p_\varepsilon}\cup v^{p_\zeta}$. It may happen that there are $i<n^\ast$ and $j<m^\ast$ such that $x^\varepsilon_i \in 
u^{p_\varepsilon}\setminus \varepsilon$ and $y^\varepsilon_j \in 
v^{p_\varepsilon}\setminus \varepsilon$, so $x^\zeta_i \in 
u^{p_\zeta}\setminus \zeta$ and $y^\zeta_j \in 
v^{p_\zeta}\setminus \zeta$, such that $c(x^\varepsilon_i, y^\varepsilon_j)$ is defined, and hence 
$c(x^\zeta_i, y^\zeta_j)$ is defined and $c(x^\zeta_i, y^\zeta_j)=c(x^\varepsilon_i, y^\varepsilon_j)$. However, for all we know, $y^\varepsilon_j$ and $y^\zeta_j$ might be compatible in $T$ and therefore we run into a problem with the requirement (\ref{mainpointQ}) of Definition \ref{defQ} of the forcing. We shall solve this difficulty by invoking the following lemma, essentially due to Baumgartner, Malitz and Reindhardt \cite{BMR}, here taken from Jech's book \cite{Jec03}, where one can find a proof. In fact, although the book states the Claim in terms of Aronszajn trees, the same proof works for any tree of height and cardinality $\omega_1$, as long as the tree does not have an uncountable branch. We shall use that fact in \S\ref{wtn}, so we state the claim in these terms.

\begin{claim}[\cite{Jec03}, Lemma 16.18]\label{BMRlemma} If $\mathbf T$ is 
tree of height and cardinality $\omega_1$ with no uncountable branches and $W$ is an uncountable collection of finite pairwise disjoint subsets of $\mathbf T$, then there exist $s, s'\in W$ such that 
any $x\in s$ is incomparable with any $y\in s'$.
\end{claim}

We can now apply Claim \ref{BMRlemma} to find $\varepsilon<\zeta$ both in $S$ such that 
any $y^\varepsilon_j$ is incomparable with any $y^\zeta_{j'}$. Now we claim that $p_\varepsilon$ and $p_\zeta$
are compatible. Let us start by defining $v=v^{p_\varepsilon} \cup v^{p_\zeta}$ and 
$u'=u^{p_\varepsilon} \cup u^{p_\zeta}$. In order to get a condition we shall have to extend $u'$ and also define 
$<$, but note already that if $\alpha\in v$, then there is an element of height $\hgt(\alpha)$ in $u'$, since the analogue is true about $u^{p_\varepsilon}$ and $u^{p_\zeta}$. So conditions 1. and 2. of Definition \ref{defQ} are easy to
fulfil and it is condition 4. that is difficult. Once we fulfil it, that Condition 3. will follow from the proof.

Our choices so far imply that $c=c^{p_\varepsilon} \cup c^{p_\zeta}$ is a well defined function. 
In order to use it to fulfil condition 4. of Definition \ref{defQ}, we have to check through all the pairs $(x_1, y_1)\neq (x_2, y_2)$ in $\bigcup_{\delta\mbox{ limit }<\omega_1} \lev_\delta(u')\times \lev_\delta(v)$  such that $c(x_1, y_1)=c(x_2, y_2)$. If $(x_1, y_1), (x_2, y_2)$ are both in 
$\dom(c^{p_\varepsilon})$ or both are in $\dom(c^{p_\zeta})$, then the condition 4. is satisfied for them, so the interesting case is when they are not. 

Therefore $\alpha(x_1, y_1)\neq \alpha(x_2, y_2)$, and let us suppose, without loss of generality, that $\alpha(x_1, y_1)< \alpha(x_2, y_2)$.
Then necessarily $(x_1, y_1)\in \dom(c^{p_\varepsilon})
\setminus \dom(c^{p_\zeta})$ and
$(x_2, y_2)\in \dom(c^{p_\zeta})\setminus  \dom(c^{p_\varepsilon})$. We have assured that this implies that 
$y_1$ and $y_2$ are incompatible in $T$. Let $\gamma=\hgt(y_1\cap_T y_2)$, so $\gamma<\alpha(x_1, y_1)$.
So far we know nothing about $x_1\cap x_2$ since neither $<_{p_\varepsilon}$ nor $<_{p_\zeta}$ have the pair $(x_1, x_2)$ in its domain. Knowing that $\alpha(x_1, y_1)$ is a limit ordinal, we are going to choose a successor
ordinal $\beta_{x_1, x_2}$
above $\max(\gamma, \gamma^\ast)$ and below $\alpha(x_1, y_1)$ and an ordinal $w_{x_1, x_2}$ of height 
$\beta_{x_1, x_2}$ which is not $<_{p_\varepsilon}$ above any element of $u^{p^\varepsilon}$. We shall add 
$w_{x_1, x_2}$ to $u'$ and declare $w_{x_1, x_2}=x_1\cap_{<} x_2$. We do this for all
pairs relevant to condition 4., by induction on the number of such pairs, each time avoiding all interaction with what we have already chosen. At the end let $u$ be the union of $u'$ and the set of all such $w_{x_1, x_2}$.
Since the new elements are all of successor height, this will not bring us in danger of creating new instances of
condition 4. Finally, to fulfil condition 3. we need to extend $<_{p_\varepsilon} \cup <_{p_\zeta}$ to a partial order 
$<$ on $u$ which will respect the commitments on $\cap_{<}$ which we have just made, which is possible by the way we chose $\beta_{x_1, x_2}$.

Then the condition $q=(u, v, <, c)$ is a common extension of $p_\varepsilon, p_\zeta$.
\end{proof}

\begin{proof} (of Theorem \ref{nounivAron}) To finish the proof, we suppose that we are in a model of $MA(\omega_1)$ and that $T$ is an Aronszajn tree.  Without generality, passing to a weakly bi-embeddable copy and adding a root if necessary, we can assume that $T$ is rooted and normal.Then by forcing by the ccc forcing $\mathbb Q (T)$ (Lemma \ref{ccc}) and intersecting 
$\aleph_1$ many dense sets $\DD_\alpha$ for $\alpha <\omega_1$ (Claim \ref{density1}) and $\EE_y$ for $y\in T$ (Claim \ref{density2}),
we obtain that the generic Aronszajn tree $T^\ast$ does not weakly embed into $T$ (Lemma \ref{genericc} and Claim \ref{specialisation}(2)). Therefore, $T$ is not universal, and since $T$ is arbitrary, the theorem is proved.
\end{proof}

\begin{remark}\label{} Theorem \ref{nounivAron} gives another proof of the main result of \cite{BMR}, which is that under $MA(\omega_1)$ all Aronszajn trees are special and another proof of Theorem \ref{Toto}.
\end{remark}

\section{Embedding wide Aronszajn trees into Aronszajn trees}\label{wtn} This section is devoted to the proof of the following theorem:

\begin{theorem}\label{wideintonormal} For every tree $T\in\mathcal \TT$, there is a ccc forcing $\mathbb P=\mathbb P(T)$
and a family $\HH=\HH(T)$ of $\aleph_1$ many dense sets in  $\mathbb P$, such that 
every $\HH$-generic filter adds a
tree in $\mathcal A$ into which $T$ weakly embeds. In particular, under the assumption of $MA(\omega_1)$,
the class $\mathcal A$ is cofinal in the class $(\TT, \le)$.
\end{theorem}

Following the pattern from Section \S\ref{sec:Aron}, we shall break the proof into the definition of the forcing and then several lemmas needed to make the desired conclusion. The forcing is dual to the one in 
 \S\ref{sec:Aron}, in the sense that we now start with a tree $T$ in  $\TT$ and generically add an Aronszajn tree that $T$ weakly embeds to. We use the control function $c$ to make sure that the generic tree does not 
 have an uncountable branch. 

For the definition of the forcing, we represent every $T\in \TT$ by an isomorphic copy which is a subtree of
$ {}^{\omega_1>} \omega_1$.

\begin{definition}\label{defP}
Suppose that $T \subseteq {}^{\omega_1>} \omega_1$ is a tree of size $\aleph_1$ and with no uncountable branches, we define a forcing notion 
$\mathbb P=\mathbb P(T)$ to consist of all $p=(u^p, v^p, <_p, f^p, c^p)$
such that:
\begin{enumerate}
\item $u^p\subseteq T$, $v^p\subseteq \omega_1$ are finite and $\langle\rangle\in u^p$,
\item $u^p$ is closed under intersections,
\item $<_p$ is a partial order on $v^p$,
\item $f^p$ is a surjective weak embedding from $(u^p,\subset)$ onto $(v^p, <_p)$,
%
%
\item for every $\eta\in u^p$, we have $\hgt(f^p(\eta))=\lg(\eta)$ (notice that $\lg(\eta)=\hgt_T(\eta)$, since $\eta$ is a sequence of ordinals),
\item\label{mainpoint} $c^p$ is a function from $v^p$ into $\omega$ such that 
 \[
\alpha<_p\beta\implies c^p(\alpha)\neq c^p(\beta).
 \]
\end{enumerate}

The order $p\le q$ on $\mathbb P$ is given by inclusion: $u^p\subseteq u^q, v^p\subseteq v^q,
<_p\subseteq <_q$ and $c^p\subseteq c^q$.

\end{definition}

\begin{lemma}\label{genericcp} There is a family $\HH$ of $\aleph_1$-many dense subsets of $\mathbb P$
such that for any $G$ which is $\HH$-generic, letting
\[
T^\ast=\bigcup\{<_p:\,p\in G\}, f= \bigcup\{f^p:\,p\in G\}, \mbox{ and }  c=\bigcup\{c^p:\,p\in G\},
\]
we have that $T^\ast$ is an Aronszajn tree, $f$ is a level-preserving weak embedding of $T$ into $T^\ast$,
$c:\,T^{\ast}\into \omega$
and $\alpha<_{T^{\ast}} \beta \implies c(\alpha)\neq c(\beta)$ .
\end{lemma}

\begin{proof} Clearly, for any filter $G$ the set $T^\ast$ is a partial order on a subset of $\omega_1$, $c$ is a well defined function into
$\omega$ and $f$ is a function from a subset of $T$ into $T^\ast$ which is a weak embedding of its domain
into its range. In addition, $f$ is level-preserving in the sense that for all $\eta\in \dom(f)$ we have
 $\hgt(f(\eta))=\lg(\eta)$ and $c$ satisfies $\alpha<_{T^{\ast}} \beta \implies c(\alpha)\neq c(\beta)$. To finish the proof of the Lemma, we prove the following three claims.

\begin{claim}[Density Claim]\label{domain} There is a set $\HH$ of $\aleph_1$ many dense subsets of $\mathbb P$
such that if $G$ is $\HH$-generic, then
domain of $f^\ast$ is $T$.
\end{claim}

\begin{proof} Let $\rho\in T$, we shall show that $\EE_\rho=\{p\in \mathbb P:\, \rho\in\dom(f^p)\}$ is dense.
Suppose that $p\in P$ is given and suppose that $p\not\in \EE_\rho$. We shall define an
extension $q$ of $p$ which is in $\EE_\rho$. Let us define $u^q_0=u^p\cup\{\rho\}$. Let $\alpha=\lg(\rho)$.  We shall first extend $f^p$ to $u^q_0$. For the ease of reading, we divide the proof into steps.

\smallskip

{\noindent (1)} The first case is that either there is no $\tau\in u^p$ with
$\rho\subset\tau$, or that there are such $\tau$ but there is no $\tau', \rho'\in u^p$
such that $\lg(\rho')=\alpha$,
$\rho'\subset \tau'$ and $f^p(\tau')=f^p(\tau)$.
In this case choose $\gamma\in  [\omega\alpha, \omega\alpha+\omega)\setminus v^p$ and define $v^q_0=v^p\cup \{\gamma\}$, $f^q(\rho)=\gamma$. Let $\gamma>_q\beta$ for any $\beta=f^p(\sigma)$ for
some $\sigma\subset \rho$ and $\gamma<_q\delta$ for any $\delta=f^p(\tau)$ for $\rho\subset \tau$ and
$\tau\in u^p$. Then the relation $<_q$ is a partial order. We let $c^q(\gamma)$ be 
any value in $\omega$ not taken by $c^p$. 

\smallskip

{\noindent (2)} This step is {\bf the main point}.
It is that there is $\tau\in u^p$ with
$\rho\subset\tau$ and $\tau', \rho'\in u^p$ such that $\lg(\rho')=\alpha$,
$\rho'\subset \tau'$ and $f^p(\tau')=f^p(\tau)$. In this case we shall have $v^q_0=v^p, <_q^0=<_p$ and $c^q=c^p$, so let us show how to extend $f^p$ to $f^q$. Let $\tau$ be of the least length among all $\tau$s as in the assumption of this case.
We are then obliged to let $f^q(\rho)= f^p(\rho')$, since
$f^p(\tau)$ can have only one restriction to the level $\alpha$ and $f^p(\rho')$ is already such a restriction.
Note that for any $\tau'', \rho'' \in u^p$ such that $\lg(\rho'')=\alpha$, $\rho''\subset \tau''$, $f^p(\tau'')=f^p(\tau)$, we must have $f^p(\rho'')=f^p(\rho')$ since $f^p$ is a weak embedding. However, there is a {\em possible problem}:
there could be $\sigma, \sigma'$ and $\rho''$ such that  $\lg(\rho'')=\alpha$,  $\rho\subset \sigma$, $\rho''\subset \sigma'$, 
$f^p(\sigma)=f^p(\sigma')$, which would force us to have $f^p(\rho)= f^p(\rho'')$, but maybe $f^p(\rho'')\neq  f^p(\rho')$. Luckily, this cannot happen since {\em $u^p$ is closed under intersections}, so for any such $\sigma$
we would have $\rho=\sigma\cap\tau\in u^p$, which is not the case. In fact, any $\sigma\in u^p$ with 
$\rho\subset \sigma$ must satisfy $\tau\subseteq \sigma$. 

\

\begin{center}
\begin{tikzpicture}
\node at (4.25,0) {Main point}; 
\draw[line width=0.5, -] (0,1) -- (2,1);
\node at (1,0.7) {${\rm lev}_\alpha(T)$}; 
\node at (4,1) {\scriptsize $\bullet$}; 
\node at (6,1) {\scriptsize $\bullet$}; 
\node at (8,1) {\scriptsize $\bullet$}; 
\node at (3.5,3) {\scriptsize $\bullet$}; 
\node at (5,3) {\scriptsize $\bullet$}; 
\node at (7,3) {\scriptsize $\bullet$}; 
\node at (8.5,3) {\scriptsize $\bullet$}; 
\draw[line width=0.5, -] (4,1) -- (3.5,3);
\draw[line width=0.5, -] (6,1) -- (5,3);
\draw[line width=0.5, -] (6,1) -- (7,3);
\draw[line width=0.5, -] (8,1) -- (8.5,3);
\draw[line width=0.5, dotted] (3.5,3) to[out=45, in=135]  (7,3);
\draw[line width=0.5, dotted] (5,3) to[out=45, in=135]  (8.5,3);
\node at (4,0.7) {$\rho'$}; 
\node at (6,0.7) {$\rho$}; 
\node at (8,0.7) {$\rho''$}; 
\node at (3.2,3) {$\tau'$}; 
\node at (4.7,3) {$\sigma$}; 
\node at (7.3,3) {$\tau$}; 
\node at (8.8,3) {$\sigma'$}; 
\end{tikzpicture}
\end{center}

\smallskip

{\noindent (3)} Now we know what $f^q(\rho)$ is and we have to discuss the closure under intersections. If there is $\tau\in u^p$ with $\rho\subset \tau$,
then taking such $\tau$ of minimal length, we have that for every $\sigma\in u^p$, $\rho\cap\sigma=\tau\cap\sigma$, by the minimality of the length of $\tau$ and the fact that $u^p$ is closed under intersections. In this case we let $u^q=u^q_0$ and $v^q=v^q_0$ and we are done. So suppose that there is no such $\tau$. Let $\sigma\in u^p$ be the longest initial segment of $\rho$ which is in $u^p$, which exists since
$u^p$ is finite and it contains $\langle\rangle$. Then, if there are intersections of the elements of $u^q_0$
which are not already be in $u^q_0$, they must be of the form $\tau\cap\rho$ for some $\tau\in u^p$ with
$\sigma\subset \tau$. Moreover, by the closure of $u^p$ under intersections and the choice of $\sigma$, there is a single 
$\tau\in u^p$ of least length which satisfies $\tau\subseteq \tau'$ for any other such $\tau'$.
We then add $\tau\cap\rho$ to $u^q_0$ to form $u^q$ and we note that this 
set is now closed under intersections. If $u^q=u^q_0$, then we are done. Otherwise, $u^q\setminus u^q_0$ is a singleton
and let $\beta$ be such that the unique element of $u^q\setminus u^q_0$ of length $\beta$.
We then choose an ordinal $\gamma_\beta\in [\omega, \omega\beta+\omega)\setminus \ran(f^p)$ and we let $f^p(\sigma)<_q \gamma_{\beta} 
<_q f^q(\rho)$. We extend $<_q$ by transitivity. Finally we
choose an element $c_\beta\in \omega\setminus \ran(c^q_0)$ and let $c^q(\gamma_\beta)=
c_\beta$.

To finish the proof of the claim, let $\HH$ consist of all $\mathcal E_\rho$ for $\rho\in T$.
\end{proof}

\begin{claim}\label{levels}
For every $\alpha<\omega_1$ we have that $\lev_\alpha(T^\ast)\subseteq [\omega\alpha, 
\omega\alpha +\omega)$ and $T^\ast$ has size 
$\aleph_1$. 
\end{claim}

\begin{proof} It follows from the definition of the forcing that 
\[
\ran(f^p\rest (\lev_\alpha(T))\subseteq [\omega\alpha, 
\omega\alpha +\omega) 
\]
for every $p\in {\mathbb P}$. That every $\lev_\alpha(T)$ is non-empty follows 
from Claim \ref{domain}.
\end{proof}

We can conclude that $T^\ast$ is an $\omega_1$-tree. By genericity we have that the domain of $c$ is $T^\ast$ and that $c:\,T^\ast\to\omega$ satisfies $\alpha<_p\beta\implies c^p(\alpha)\neq c^p(\beta)$. 

\begin{claim}\label{cbranches} $T^\ast$ has no uncountable branch.
\end{claim}

\begin{proof} This is an easy consequence of the properties of $c$, namely $c$ is 1-1 on any branch, and its range is a subset of $\omega$.
\end{proof}

Therefore $T^\ast$ is an Aronszajn tree. To finish the proof of the lemma, it remains to verify that 
$f:\,T\into T^\ast$ is a weak embedding, which follows from the genericity.
\end{proof}

\begin{lemma}\label{cccp} The forcing $\mathbb P(T)$ is ccc.
\end{lemma}

\begin{proof} Recalling that the elements $\rho$ of $T$ are functions from a countable ordinal to $\omega_1$, we shall use the notation $\rho\rest\alpha$ to denote the restriction of $\rho$ to $\lg(\rho)\cap\alpha=\max\{\lg(\rho), \alpha\}$ 
and $\rho(\beta)$ for the value of $\rho$ at $\beta\in\lg(\rho)$. Also observe that $\lg(\rho)=\dom(\rho)$.

Suppose that $\langle p_\zeta:\,\zeta <\omega_1\rangle$ is a given sequence of elements of
$\mathbb P(T)$. By extending each $p_\zeta$ if necessary, using the density of the sets $\EE_\rho$ from Claim 
\ref{domain}, we can assume that for each $\zeta<\omega_1$:
\begin{description}
\item[(a)] 
there is
an element of $u^{p_\zeta}$ and hence of  $v^{p_\zeta}$ of height $\zeta$, and that
\item[(b)]\label{abb} for every $\rho\in 
u^{p_\zeta}$ and every $\beta<\lg(\rho)$ such that there is an element of $u^{p_\zeta}$ of height $\beta$,
the point $\rho\rest\beta$ is in $u^{p_\zeta}$. 
\end{description}
Let 
\[
C=\big\{\zeta <\omega_1:\, \omega\zeta=\zeta\mbox{ and } \max\{ \lg(\rho), \rho(\alpha): \, \rho \in \bigcup_{\varepsilon <\zeta } u^{p_\varepsilon}, \alpha< \lg(\rho) \} \ < \zeta\big\},
\]
so $C$ is a club of $\omega_1$ consisting of limit ordinals. By extending again if necessary, we shall require that for every
$\zeta\in C$, there is an element in $u^{p_\zeta}$ of height in $(0,\zeta)$.
For $\zeta\in C$ let us define $q_\zeta=p_\zeta\rest \zeta$, by which we mean:
\begin{enumerate}
\item $u^{q_\zeta}=u^{p_\zeta}\cap {}^{<\zeta}\omega_1$, $v^{q_\zeta}=v^{p_\zeta}\cap \zeta$, 
\item $<_{q_\zeta}=<_{p_\zeta}\rest v^{q_\zeta}$ and
\item $f^{q^\zeta}=f^{p^\zeta}\rest u^{q_\zeta}$, $c^{q_\zeta}=c^{p_\zeta} \rest v^{q_\zeta}$.
\end{enumerate}
Applying the Fodor Lemma and the Delta-System Lemma, we obtain a stationary set $S\subseteq C$ such that: 
\begin{enumerate}
\item  for every 
$\zeta\in S$ we have:
$v^{q_\zeta}=v^\ast$, $<_{q_\zeta}=<^\ast$, $c^{q_\zeta}=c^\ast$ are fixed,
\item the sets $u^{q_\zeta}$ form a $\Delta$-system with root $u^\ast$,
\item for every $\varepsilon< \zeta\in S$ there is a level-preserving order isomorphism $\varphi_{\varepsilon,\zeta}$
from $u^{q_\varepsilon}$ to $u^{q_\zeta}$ which is identity on $u^\ast$,\footnote{Since $u^{p_\varepsilon}$ and $u^{p_\zeta}$ are closed under intersections, $\varphi_{\varepsilon,\zeta}$ necessarily preserves intersections.}
\item for every $\varepsilon< \zeta\in S$, $f^{q_\varepsilon}=f^{q_\zeta}\circ \varphi_{\varepsilon,\zeta}$,
\item for every $\varepsilon <\zeta\in S$, there is an order preserving isomorphism $\psi_{\varepsilon,\zeta}$ from $(u^{p_\varepsilon}, \subseteq)$ to $(u^{p_\zeta}, \subseteq)$ which extends 
$\varphi_{\varepsilon,\zeta}$ and such that $f^{p_\epsilon}=f^{p_\zeta}\circ \psi_{\varepsilon,\zeta}$,
\item for every $\varepsilon< \zeta\in S$, there is an order preserving isomorphism $i_{\varepsilon,\zeta}$ from 
$(v^{p_\varepsilon}, <_{p_\varepsilon})$ to $(v^{p_\zeta}, <_{q_\zeta})$ which is identity on $v^\ast$.
\end{enumerate}

By the fact that there is an element of height $\zeta$ in $u^{p_\zeta}$, we have that each $u^{p_\zeta}\setminus u^{q_\zeta}\neq\emptyset$. 
Since $\langle\rangle\in u^{q_\zeta}$ we have that $u^{q_\zeta}\neq \emptyset$ for all $\zeta$, but even more so,
$u^{q_\zeta}$ has an element of height in $(0, \zeta)$. Let
$\alpha_1=\max\{\lg(\rho):\,\rho\in u^{q_\zeta}\}$ and $\alpha_0=\min\{\lg(\rho):\,\rho\neq \langle\rangle \in u^{q_\zeta}\}$. 
Since $\zeta$ is an element of $C$, it is a limit ordinal. The set $u^{q_\zeta}$ is finite set and for every $\rho\in u^{q_\zeta}$
the length $\lg(rho)<\zeta$, so we have that $\alpha_1<\zeta$.
By the choice of $\varphi_{\varepsilon,\zeta}$, the choice of $\alpha_0$ and $\alpha_1$ does not depend on $\zeta$.
Finally let $\delta=\min (C)\setminus \alpha_1$.

Our requirements and the fact that $T$ does not have an uncountable branch imply that we can use Claim \ref{BMRlemma} to find $\varepsilon< \zeta\in S\setminus \delta$ such that for every 
$\rho\in u^{p_\varepsilon}\setminus u^\ast$ and  $\sigma \in u^{p_\zeta}\setminus u^\ast$, 
$\rho$ and $\sigma$ are incomparable. We shall find a common extension of $p_\varepsilon$ and $p_\zeta$.

We first define $u_0=u^{p_\varepsilon} \cup u^{p_\zeta}$. We also define $f_0=f^{p_\varepsilon}\cup f^{p_\zeta}$, which is well defined by the assumptions of the $\Delta$-system, and, similarly, $c_0=c^{p_\varepsilon}\cup c^{p_\zeta}$. We also simply let $<_0=<_{p_\varepsilon}\cup <_{p_\zeta}$, which still gives a partial order by the choice of $\varepsilon$ and $\zeta$. Specifically, $<_0$ makes any element of $v^{p_\varepsilon}$ incomparable to any element of and $v^{p_\zeta}$, which conforms to the fact that any element of $u^{p_\varepsilon}\setminus u^\ast$ is incomparable to any element of $u^{p_\zeta}\setminus u^\ast$.

The only problem is that $u_0$ is not necessarily closed under intersections. Let us analyse what type of intersection can occur and what we need to add to make
$u_0$ closed under
intersections.

Let $\rho,\tau\in u_0$. If 
$\rho,\tau\in u^{p_\varepsilon}$ or $\rho,\tau\in u^{p_\zeta}$ then $\rho\cap\tau \in u^\ast$. Let us now suppose that 
we are dealing with some $\rho\in u^{p_\varepsilon}\setminus u^{p_\zeta}$ and $\tau\in u^{p_\zeta}\setminus u^{p_\varepsilon}$, the other case is symmetric. 

\smallskip

{\noindent\underline{Case 1.}} $\lg(\rho\cap \tau) <\alpha_0$. We handle all instances of such $\rho$ and $\tau$ simultaneously.


Using that $\rho\rest \alpha_0\in u^{p_\varepsilon}$
and $\tau\rest \alpha_0\in u^{p_\zeta}$, it suffices to consider the case $\lg(\rho)=\lg(\tau)=\alpha_0$.

Let $\sigma_0, \ldots, \sigma_n$ be all $\sigma=\rho\cap\tau$ obtained in this way. 
We choose for each $i< n+1$ distinct $f(\sigma_i)$ with $\hgt(f(\sigma_i))=\lg(\sigma_i)$ (note that necessarily  $f(\sigma_i) \in \omega_1\setminus \ran(f_0)$) and distinct $c_i$ in
$\omega\setminus \ran(c_0)$. Extend $u_0$ by adding all $\sigma_0, \ldots, \sigma_n$ and $v_0$
by adding all $f(\sigma_i)$. Extend $<_0$ to a transitive order on $v_0$ which satisfies $f(\sigma_i)<_0
f(\eta)$ when $\sigma_i\subset \eta$ for some $\rho\in u^{p_\varepsilon}\cup u^{p_\zeta}$. This is possible because
there are no elements of $u_0$ of length $<\alpha_0$.
Extend $c_0$ to include the values $c_i=c(f(\sigma_i))$ as above. Call the resulting tuple
$(u_1, v_1, <_1, f_1, c_1)$.

\smallskip

{\noindent\underline{Case 2.}} $\lg(\rho\cap \tau) \in [\alpha_0, \alpha_1)$. We handle all instances of such $\rho$ and $\tau$ simultaneously.

Let 
$\sigma=\rho\cap \tau$.
By our assumption {\bf (b)} we can assume that 
$\rho\in u^{q_\varepsilon} \setminus u^\ast$ and $\tau \in u^{q_\zeta}\setminus u^\ast$ are of the least possible length with the intersection $\sigma$.  By the fact that $\varphi_{\varepsilon, \zeta}$ preserves both order and height, another application of {\bf (b)} lets us assume that $\hgt(\rho)=\hgt(\tau)$.
The possible {\bf dangerous configuration} is that there are $\rho'\in 
u^{q_\varepsilon} \setminus u^\ast$, $\tau'\in 
u^{q_\zeta} \setminus u^\ast$ of length $\lg(\rho)$ and $\sigma'\in u^{q_\varepsilon} \setminus u^\ast$,
$\sigma''\in u^{q_\zeta} \setminus u^\ast$ of length $\lg(\sigma)$ such that $\sigma'\subset \rho'$ 
and $\sigma''\subset\tau'$, $f^{p_\varepsilon}(\rho')=f^{p_\varepsilon}(\rho)=f^{p_\zeta}(\tau)=
f^{p_\zeta}(\tau')$, yet  $f^{p_\varepsilon}(\sigma')\neq f^{p_\zeta}(\sigma'')$. 

\

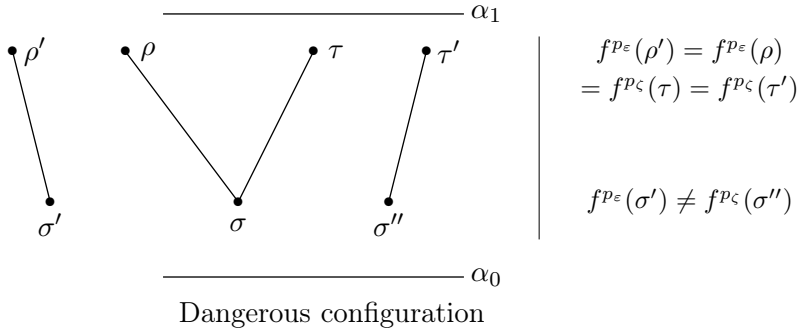
\begin{center}
\begin{tikzpicture}
\node at (4.25,0) {Dangerous configuration}; 
\draw[line width=0.5, -] (2,0.5) -- (6,0.5);
\draw[line width=0.5, -] (2,4) -- (6,4);
\node at (6.3,0.5) {$\alpha_0$}; 
\node at (6.3,4) {$\alpha_1$}; 
\node at (0.5,1.5) {\scriptsize $\bullet$}; 
\node at (3,1.5) {\scriptsize $\bullet$}; 
\node at (5,1.5) {\scriptsize $\bullet$}; 
\node at (0,3.5) {\scriptsize $\bullet$}; 
\node at (1.5,3.5) {\scriptsize $\bullet$}; 
\node at (4,3.5) {\scriptsize $\bullet$}; 
\node at (5.5,3.5) {\scriptsize $\bullet$}; 
\draw[line width=0.5, -] (0.5,1.5) -- (0,3.5);
\draw[line width=0.5, -] (3,1.5) -- (1.5,3.5);
\draw[line width=0.5, -] (3,1.5) -- (4,3.5);
\draw[line width=0.5, -] (5,1.5) -- (5.5,3.5);
\node at (0.5,1.2) {$\sigma'$}; 
\node at (3,1.2) {$\sigma$}; 
\node at (5,1.2) {$\sigma''$}; 
\node at (0.3,3.5) {$\rho'$}; 
\node at (1.8,3.5) {$\rho$}; 
\node at (4.3,3.5) {$\tau$}; 
\node at (5.8,3.5) {$\tau'$}; 
\draw[line width=0.1, -] (7,1) -- (7,3.7);
\node at (9,1.5) {\small $f^{p_\varepsilon}(\sigma') \neq f^{p_\zeta}(\sigma'')$}; 
\node at (9,3.5) {\small $f^{p_\varepsilon}(\rho') = f^{p_\varepsilon}(\rho)$}; 
\node at (9,3) {\small $= f^{p_\zeta}(\tau)= f^{p_\zeta}(\tau')$}; 
\end{tikzpicture}
\end{center}

If there were such points we
would not be able to extend $f_1$ to $\sigma$ and keep it a weak embedding. Luckily, this cannot happen since
if there were to be any elements $\eta$ of $u^{q_\varepsilon}$ of length $\lg(\sigma)$, then by the fact that 
$u^{p_\varepsilon}$ satisfies the assumption {\bf (b)}, $\sigma=\rho\rest\lg(\eta)$
would already be $u^{p_\varepsilon}$, so in $u^{q_\varepsilon}$.

This analysis shows that we can proceed as in Case 1 to extend $(u_1, v_1, <_1, f_1, c_1)$ to $(u_2, v_2, <_2, f_2, c_2)$ which is closed under all intersections of Case 2 and satisfies other requirements of being a condition. Note that $(u_2, v_2, <_2, f_2, c_2)$ remains closed under the intersections of length $<\alpha_0$.

{\noindent\underline{Case 3.}} $\lg(\rho\cap \tau) =\alpha_1$.

Let 
$\sigma=\rho\cap \tau$.
We have that $\sigma=\rho\rest\alpha_1\in u^{p_\varepsilon}$ and $\sigma=\tau\rest\alpha_1\in u^{p_\zeta}$
and hence $\sigma\in u^\ast$, a contradiction.

{\noindent\underline{Case 4.}} $\lg(\rho\cap \tau) >\alpha_1$.

Let 
$\sigma=\rho\cap \tau$.
By the choice of $S$, we have that $u^{p_\zeta}$ does not have any elements of length $\lg(\sigma)$
and by the fact that $u^{p_\varepsilon}$ is closed under restrictions, since $\sigma=\rho\rest \lg(\sigma)$,
we have that there are no elements of $u^{p_\varepsilon}$ of length $\lg(\sigma)$ either. Hence we can proceed like in Case 1. 
Once we are done closing under intersections of this type, we finally obtain 
a common extension of $p_\varepsilon, p_\zeta$.
\end{proof}

\begin{proof} (of Theorem \ref{wideintonormal}) The proof follows by putting the lemmas together.
\end{proof}

\section{Conclusion}\label{concl} Putting the results of Section \S\ref{sec:Aron} and Section \S\ref{wtn} together, we obtain our
main theorem, as follows.

\begin{theorem}\label{main} Under $MA(\omega_1)$, there is no wide Aronszajn tree universal under weak embeddings.
\end{theorem}

\begin{proof} Assume $MA(\omega_1)$ and suppose for a contradiction that $T$ is a universal element in
$(\TT,\le)$. By Theorem \ref{wideintonormal}, there is an Aronszajn tree $T'$ such that $T\le T'$, so $T'$ is
universal in $(\TT,\le)$ and so in  $(\mathcal A,\le)$. However, by Theorem \ref{nounivAron} $(\mathcal A,\le)$ does not have a universal element, a contradiction.
\end{proof}

We also remark that putting our results together with the results of Todor{\v c}evi\'c mentioned in \S\ref{}, gives the first part of the following Corollary \ref{improve}. The second part of the corollary improves Todor{\v c}evi\'c's theorem \ref{Toto} and our Main Theorem \ref{main}.

\begin{corollary}\label{improve} Assume $MA(\omega_1)$. Then:

{\noindent (1)} The class $\LL$ of Lipschitz trees and the class of 
coherent trees are cofinal in the class $(\TT,\le)$.

{\noindent (2)} There is no element of $(\TT,\le)$ which suffices to weakly embed all Aronszajn trees
\end{corollary}

\begin{proof} (1) Our Main Theorem \ref{main} shows that under 
$MA(\omega_1)$, the class or Aronszajn trees is cofinal in the class of wide Aronszajn trees. On the other hand, 
Todor{\v c}evi\'c in \S4 of \cite{stevo2007walks} has shown that under the same assumptions, the class of coherent trees
is cofinal in the class of all Aronszajn trees and that every coherent tree is Lipschitz.

{\noindent (2)} This is a direct consequence of Theorem \ref{wideintonormal}.
\end{proof}

\bibliographystyle{plain}
\bibliography{../bibliomaster}

\end{document}